\documentclass[12pt, oneside]{scrartcl}
\usepackage[english]{babel}
\usepackage[T1]{fontenc}
\usepackage{amssymb}
\usepackage{amsmath}
\usepackage{hhline}
\usepackage{longtable}
\usepackage{amscd}
\usepackage{theorem}
\usepackage{array}
\usepackage{delarray}
\usepackage{multicol}
\usepackage{makeidx}

\usepackage{float}

\usepackage{tikz}
\usepackage{calc}
\usetikzlibrary{angles,quotes}
\usetikzlibrary{calc}
\usetikzlibrary{patterns}
\usetikzlibrary{graphs}
\usetikzlibrary{graphs.standard}
\usetikzlibrary{decorations.markings,shapes.geometric,positioning}

\newtheorem{theorem}{Theorem}

\newtheorem{lemma}{Lemma}
\newtheorem{claim}{Claim}[theorem]

\newtheorem{problem}{Problem}

\newenvironment{proof}[1][Proof]{\textbf{#1.} }{\ \rule{0.5em}{0.5em}}

\newcommand{\dist}{\mathrm{dist}}
\newcommand{\mad}{\mathrm{mad}}

\usepackage{tikz}
\author{Maidoun Mortada\footnote{KALMA Laboratory, Faculty of Sciences, Lebanese University, Baalbek, Lebanon.}, Olivier Togni\footnote{LIB Laboratory, Université de Bourgogne, Dijon, France.}}
    
\begin{document} 
\title{Further Results and Questions on $S$-Packing Coloring of Subcubic Graphs}
\maketitle

\begin{abstract}
For non-decreasing sequence  of integers $S=(a_1,a_2, \dots, a_k)$, an $S$-packing coloring of $G$ is a partition of $V(G)$ into $k$ subsets $V_1,V_2,\dots,V_k$ such that  the distance between any two distinct vertices $x,y \in V_i$ is at least $a_{i}+1$, $1\leq i\leq k$. We consider the $S$-packing coloring problem on subclasses of subcubic graphs: For $0\le i\le 3$, a subcubic graph $G$ is said to be $i$-saturated if every vertex of degree 3 is adjacent to at most $i$ vertices of degree 3. Furthermore, a vertex of degree 3 in a subcubic graph is called heavy if all its three neighbors are of degree 3, and $G$ is said to be $(3,i)$-saturated if every heavy vertex is adjacent to at most $i$ heavy vertices. We prove that every 1-saturated subcubic graph is $(1,1,3,3)$-packing colorable and  $(1,2,2,2,2)$-packing colorable. We also prove that every $(3,0)$-saturated subcubic graph is $(1,2,2,2,2,2)$-packing colorable. 
\end{abstract}

\textbf{Keywords}: graph, coloring, $S$-packing coloring, packing chromatic number, cubic graph, saturated subcubic graph.\\

\section{Introduction}
All graphs considered here are simple graphs. For a graph $G$, the set of vertices of $G$ is denoted by $V(G)$ and its set of edges by  $E(G)$. For a vertex $x\in V(G)$,  $N(x)$ denotes the set of neighbors of $x$ and $d(x)$ denotes the number of neighbors of $x$.  We denote by $\Delta(G)$ and $\delta(G)$ the maximum and minimum degree, respectively. For two vertices $x$ and $y$ in $G$, we say $y$ is a {\em second neighbor} of $x$ if $y\notin N(x)$ but $y\in N(z)$ for some $z\in N(x)$.   A graph $G$ is {\em subcubic} if $\Delta(G)\leq 3$ and {\em cubic} if for any vertex $x$, $d(x)=3$. A vertex $x$ in a subcubic graph is said to be an $i$-vertex, $0\leq i\leq 3$, if $d(x)=i$. Let $H\subseteq  V(G)$, we denote by $G[H]$ the subgraph induced by $H$. For a path $P=x_1 \dots x_n$, we call $x_1$ and $x_n$ the {\em ends} of $P$, while each other vertex is called an {\em interior} vertex. A path $P$ in a graph $G$ is said to be maximal if $P$ is not a subpath of any other path in $G$.  The length of a shortest path in $G$ joining two vertices  $x$ and $y$ is the distance between  $x$ and $y$ in $G$ and it is denoted by $\dist(x,y)$. For a graph $G$, the {\em subdivision} of $G$, denoted by $S(G)$ is the graph obtained from $G$ by replacing each edge with a path of length two.

Let $G$ be a subcubic graph. For $0\le i\le 3$, a vertex $x$ in $G$ is said to be an \emph{$i$-vertex} if $d(x)=i$. A 3-vertex in $G$ is said to be a \emph{heavy} vertex if all its neighbors are 3-vertices. In \cite{26,27}, the authors classify the subcubic graphs into four classes: For $0\le i\le 3$, $G$ is said to be \emph{$i$-saturated} if every 3-vertex in $G$ is adjacent to at most $i$ 3-vertices. Note that a $0$-saturated subcubic graph is also called  3-irregular in~\cite{16,a}.
For the class of 3-saturated subcubic graphs, the authors in \cite{26} consider, for $0\le i\le 3$, the subclass of \emph{$(3,i)$-saturated} subcubic graphs, which consists of the 3-saturated subcubic graphs such that every heavy vertex is adjacent to at most $i$ heavy vertices. We point out that the above definition of saturation is slightly modified from the original one given in~\cite{26,27}, in order to have a natural inclusion scheme. Remark also that any subcubic graph is $(3,3)$-saturated.

For a sequence  of positive integers $S=(a_1,a_2,\dots, a_k)$  with $a_1\leq a_2\leq \dots \leq a_k$, an {\em $S$-packing coloring} of a graph $G$ is a partition of $V(G)$ into  subsets $V_1, V_2, \dots, V_k$ such that for every two distinct vertices $x$ and $y$  in $V_i$,  $\dist(x,y)\geq a_i+1$ for $1\le i\le k$. The smallest $k$ such that $G$ is $(1,2, \dots, k)$-packing colorable is called the \emph{packing chromatic number} of $G$ and is denoted by $\chi_{\rho}(G)$. This parameter was introduced by Goddard et al.~\cite{d} under the name of broadcast chromatic number. Since then, it has been studied extensively, see the survey paper of Bre\v{s}ar, Ferme, Klav\v{z}ar and Rall~\cite{bfk}. For better visibility, we may use exponents in a sequence to denote the repetition of an integer, e.g., $(1^2,2^3)=(1,1,2,2,2)$.

The class of subcubic graphs appears to be attractive for the packing and $S$-packing coloring problems. Balogh, Kostochka and Liu~\cite{bkl} and Bre\v{s}ar and Ferme~\cite{bf} independently proved that the packing chromatic number is not bounded on the class of subcubic graphs. Many papers~\cite{3,11,8,9,16,24,25} were devoted to finding bounds on $\chi_{\rho}(G)$ and $\chi_{\rho}(S(G))$ for subcubic graph subclasses, or finding sequences $S$ for which the graphs in the class are $S$-packing colorable.
For $S$-packing coloring of subcubic graphs, the following problems reveal to be challenging:
\begin{problem}[\cite{16,9}]
\begin{enumerate}
\item Is every subcubic graph except the Petersen graph $(1,1,2,2)$-packing colorable?
\item Is every subcubic graph except the Petersen graph $(1,2^5)$-packing colorable?
\item Does every subcubic graph $G$ satisfy $\chi_{\rho}(S(G)) \leq 5$?
\end{enumerate}

\end{problem}

  \begin{figure}[h]
\begin{center}
\begin{tikzpicture}[scale=1]
\node at (2,2) (a) [circle,draw=black,fill=none,scale=0.7]{};
\node at (2*.866,1+2) (b) [circle,draw=black,fill=none,scale=0.7]{};
\node at (1,2*.866+2) (c)[circle,draw=black,fill=none,scale=0.7]{};
\node at (0,4) (d) [circle,draw=black,fill=none,scale=0.7]{};
\node at (-1,2*.866+2) (e) [circle,draw=black,fill=none,scale=0.7]{};
\node at (-2*.866,1+2) (f)[circle,draw=black,fill=none,scale=0.7]{};
\node at (-2,2) (g) [circle,draw=black,fill=none,scale=0.7]{};
\node at (-2*.866,1) (h) [circle,draw=black,fill=none,scale=0.7]{};
\node at (-1,-2*.866+2) (i) [circle,draw=black,fill=none,scale=0.7]{};
\node at (0,0) (j) [circle,draw=black,fill=none,scale=0.7]{};
\node at (1, 2-2*.866) (k) [circle,draw=black,fill=none,scale=0.7]{};
\node at (2*.866,1) (l) [circle,draw=black,fill=none,scale=0.7]{};
\draw (a) -- (b) -- (c) -- (d) -- (e) -- (f) -- (g) -- (h) -- (i) -- (j) -- (k) -- (l) -- (a);
\draw (c) .. controls (0,3.5) .. (e);
\draw (g) .. controls(-1.2,1.2) .. (i);
\draw (a) .. controls(1.2,1.2) .. (k);

\node at (7,4-.2) (s) [circle,draw=black,fill=none,scale=0.7]{};
\node at (8,4-.2) (t) [circle,draw=black,fill=none,scale=0.7]{};
\node at ( 7.5, 3.3-.2 ) (r) [circle,draw=black,fill=none,scale=0.7]{};
\node at ( 6.1 , 1-.5 ) (v) [circle,draw=black,fill=none,scale=0.7]{};
\node at ( 5.6 , 1.7 -.5) (w) [circle,draw=black,fill=none,scale=0.7]{};
\node at ( 6.5 , 1.8-.5) (u) [circle,draw=black,fill=none,scale=0.7]{};
\node at ( 8.9 , 1-.5) (y) [circle,draw=black,fill=none,scale=0.7]{};
\node at ( 9.5 , 1.7-.5) (z) [circle,draw=black,fill=none,scale=0.7]{};
\node at ( 8.6 , 1.8-.5) (x) [circle,draw=black,fill=none,scale=0.7]{};
\node at ( 7.5, 2.1-.5) (c') [circle,draw=black,fill=none,scale=0.7]{};
\node at ( 7.5, 1-.5) (b') [circle,draw=black,fill=none,scale=0.7]{};
\node at ( 7.5, 2.7-.35) (h') [circle,draw=black,fill=none,scale=0.7]{};
\draw (s) -- (t) -- (r) -- (s);
\draw (v) -- (w) -- (u) -- (v);
\draw (x) -- (y) -- (z) -- (x);
\draw (s) -- (w);
\draw (t) -- (z);
\draw (y) -- (b') -- (v);
\draw (r) -- (h') -- (c') -- (u);
\draw (x) -- (c');
\end{tikzpicture}
\end{center}
\caption{A non $(1,1,4,4)$-packing colorable $1$-saturated graph (on the left) and a non $(1,1,3,3)$-packing colorable $(3,2)$-saturated subcubic graph (on the right).}
\label{fig2}
\end{figure}
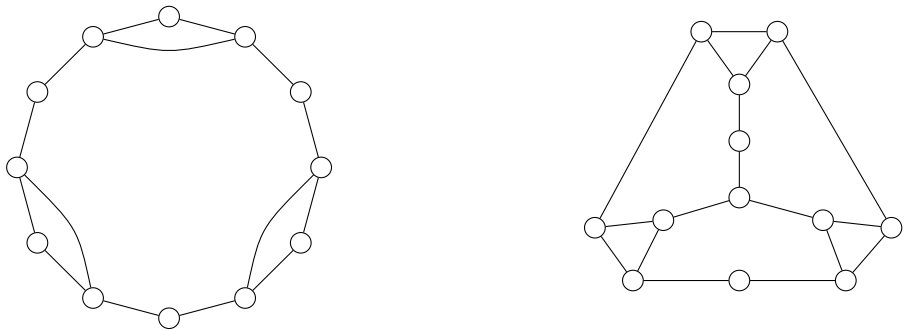

There is a natural link between Problem 1.1 and Problem 1.3 since Gastineau and Togni \cite{16} proved that in order for a subcubic graph $G$ to have $\chi_{p}(S(G)) \leq 5$ it is enough for $G$ to be $(1, 1, 2, 2)$-packing colorable. For Problem 1.3, Balogh, Kostochka and Liu~\cite{3} proved that the subdivision of any subcubic graph has packing chromatic number at most 8.  
Problem 1 was confirmed true for subclasses of subcubic graphs:  Yang and Wu \cite{a} proved that every 3-irregular subcubic graph is $(1,1,3)$-packing colorable and then a simpler proof for the same result was presented in  \cite{mai}. 
Problem 1.1 was solved for generalized prism of cycles by Bre\v{s}ar, Klav\v{z}ar, Rall and Wash~\cite{9} and for subcubic graphs of maximum average degree ($\mad$) less than $30/11$ by Liu, Liu, Rolek, and Zhu~\cite{m}.
Moreover,  Tarhini and Togni \cite{i} proved that  every cubic Halin graph is $(1, 1, 2, 3)$-packing colorable. Recently, Mortada and Togni ~\cite{26,27} proved that every 1-saturated subcubic graph is $(1,1,2)$-packing colorable, every $(3,0)$-saturated subcubic graph is  $(1,1,2,2)$-packing colorable, every 2-saturated subcubic graph is $(1,1,2,3)$-packing colorable. Finally, even more recently, Liu, Zhang and Zhang~\cite{LZZ24} proved that every subcubic graph is $(1,1,2,2,3)$-packing colorable, and hence that the packing chromatic number of the subdivision of any subcubic graph is at most $6$.

We may remark here that even if the saturation properties and the maximum average degree are somewhat linked, there are $(3,0)$-saturated graphs having maximum average degree greater than $30/11$, hence for which Liu et al.'s result~\cite{m} does not apply. For instance, let $G$ be obtained by taking the prism of a triangle and, for two edges lying in the two different triangles and not on a common 4-cycle, subdividing these two edges once. Then $G$ is $(3,0)$-saturated, has average degree $22/8$ and thus $\mad(G) \ge 22/8 > 30/11$.

 \begin{table}[h]
 \centering
\begin{tabular}{|p{3.1cm}||l|l|l|}\hline
 Subcubic class & Positive sequences & Negative sequences & Open sequences\\\hline\hline
 Arbitrary & $(1,1,2,2,3)$~\cite{LZZ24} & $(1,1,2,3)$ for $P$ & $(1,1,2,2)$\\
($(3,3)$-saturated) & $(1,2^6)$~\cite{16} & $(1,2^5)$  for $P$ & $(1,2^5)$ except $P$\\\hline
 $(3,2)$-saturated & $(1,1,2,2,3)$~\cite{LZZ24} & $\mathbf{(1,1,3,3)}$ Fig.~\ref{fig2} & $(1,1,2,2)$ \\
 & $(1,2^6)$~\cite{16} &  $\mathbf{(1,2^3)}$ Fig.~\ref{Fig1} &  $(1,2^4)$\\\hline
 $(3,1)$-saturated & $(1,1,2,2,3)$~\cite{LZZ24} & $\mathbf{(1,1,4,4)}$  Fig.~\ref{fig2} & $(1,1,2,2)$ \\
 & $(1,2^6)$~\cite{16} &  $\mathbf{(1,2^3)}$ Fig.~\ref{Fig1} &  $(1,2^4)$\\\hline
 $(3,0)$-saturated & $(1,1,2,2)$~\cite{26}& $\mathbf{(1,1,4,4)}$ Fig.~\ref{fig2}& $(1,1,2,3)$\\
  &  $\mathbf{(1,2^5)}$~Thm \ref{thm5}  &  $\mathbf{(1,2^3)}$ Fig.~\ref{Fig1} &  $(1,2^4)$\\\hline\hline
 $2$-saturated & $(1,1,2,3)$~\cite{27} & $\mathbf{(1,1,4,4)}$ Fig.~\ref{fig2} & $(1,1,2,4)$\\
 & $\mathbf{(1,2^5)}$~Thm \ref{thm5} &  $\mathbf{(1,2^3)}$ Fig.~\ref{Fig1} &  $(1,2^4)$\\\hline
 $1$-saturated & $(1,1,2)$~\cite{26}  & $(1,1,3)$ ~\cite{26}& $(1,1,3,4)$\\
                & $\mathbf{(1,1,3,3)}$ Thm~\ref{thm1} & $\mathbf{(1,1,4,4)}$ Fig.~\ref{fig2}&  \\
                & $\mathbf{(1,2^4)}$ Thm~\ref{thm4} & $\mathbf{(1,2^3)}$ Fig.~\ref{Fig:1sat} &\\\hline
 $0$-saturated  & $(1,1,3)$~\cite{a} & $\mathbf{(1,1,4)}$ for $2K_3^*$ & $(1,2,2,3)$  \\
 ($3$-irregular)& $(1,2^3)$~\cite{16}& $\mathbf{(1,2,2)}$ for $S(K_4)$& \\\hline
\end{tabular}
\caption{Known results for $S$-packing coloring of $i$-saturated subcubic graphs. The results of this paper are in bold. By a positive sequence, we mean a sequence $S$ for which every graph in the class is $S$-packing colorable; a negative sequence is a sequence $S$ for which there exists a graph in the class that is not $S$-packing colorable; an open sequence is a sequence for which we do not know if it is positive or negative. $P$ is the Petersen graph and $2K_3^*$ is the graph obtained by joining two $K_3$ by a path of length two.}
 \label{tbl1}
 \end{table}

In this paper, we continue the exploration of the $S$ packing coloring problem on subcubic graph subclasses by finding new results for saturated subcubic graphs. The technique used to prove our latest results \cite{26,27} seems powerful as it allows us in this article to prove that every 1-saturated subcubic graph is $(1,1,3,3)$-packing colorable.  The technique is based on considering an independent set in a 1-saturated subcubic graph $G$ that maximizes, among all independent sets, a linear combination of the number of 3-vertices with one neighbor of degree three, the number of 3-vertices with no neighbor of degree three,  and the number of 2-vertices. Considering such an independent set allows us to determine the distance between a sufficient number of vertices in $G$, leading at the end to the desired packing coloring of $G$.  
Moreover, we prove in Section 3 that every 1-saturated subcubic graph is $(1,2^4)$-packing colorable, and every $(3,0)$-saturated subcubic graph is $(1,2^5)$-packing colorable. 
Table~\ref{tbl1} summarizes the results and questions concerning $S$-packing coloring of the above subcubic graph subclasses.

\section{$(1,1,3,3)$-Packing Coloring of 1-Saturated Subcubic Graphs}

In this section, we prove our main result concerning the 1-saturated subcubic graphs. 

\begin{theorem}\label{thm1}
Every 1-saturated subcubic graph is $(1,1,3,3)$-packing colorable.
\end{theorem}
\begin{proof}
On the contrary, suppose that $G$ is a counter-example with the minimum order.  Clearly, $G$ is connected. First, $\delta(G)\geq 2$, since otherwise, let $u$ be a vertex of degree one and let $G'=G-u$. By the minimality of $G$, $G'$ has  a $(1, 1, 3,3)$-coloring. Either $1_a$ or $1_b$ is not the color of the unique neighbor of $u$ in $G'$, then give this color to $u$, and so we obtain a $(1,1,3,3)$-coloring of $G$, a contradiction.\\

\noindent Our plan is to partition the set of vertices of $G$ into four subsets on which two of them are independent, and any two vertices in the third (resp. fourth) subset are at distance at least four. The existence of such a partition proves that $G$ is $(1,1,3,3)$-packing colorable, which is a contradiction. To reach this partition, we will first consider a special independent set that will lead to determining the distance between specific vertices in $G$.\\

\noindent Note  that if $H$ is a subgraph of $G$ or a subset of $V(G)$ and if $x$ is a vertex in $H$, by saying $x$ is an $i$-vertex, we mean that $x$ is an $i$-vertex in $G$, $2\leq i\leq 3$. That is, maybe $x$ does not have $i$ neighbors in $H$ but has them in $G$. 

\noindent Let $T$ be an independent set in $G$, we define the following three sets partitioning $T$:\\
$X_1(T)=\{x\in T:\text{ $x$ is a 3-vertex and $x$ has a neighbor of degree three}\}$,\\
$X_0(T)=\{x\in T:\text{ $x$ is a 3-vertex and $x$ has  no neighbor of degree three}\}$, and\\
$Y(T)=\{x\in T :\text{$x$ is a 2-vertex}\}$.\\ \noindent Let $\phi(T)=|X_1(T)|+0.7|X_0(T)|+0.35|Y(T)|$ and let $\overline{T}$ denotes the set $V(G)\setminus T$.   An independent set $T$ is said to be a {\em maximum weighted} independent set if $\phi(T)\geq \phi(K)$ for every independent set $K$.
\noindent Let $S$ be a maximum weighted independent set.
 Clearly, by the maximality of $\phi(S)$, each vertex in $\overline{S}$ has a neighbor in $S$.  Thus, any interior vertex of a path in $G[\overline{S}]$ is a 3-vertex. We first present this result:

\begin{claim}\label{cl1.1}
If $u$ and $v$ are two $3$-vertices in $\overline{S}$, then $u$ and $v$ are not adjacent.
\end{claim}
\begin{proof}
Suppose $u$ and $v$ are adjacent, then each neighbor  of $u$ (resp. $v$) in $S$  is a 2-vertex. Consequently,  $S'=(S\setminus N(u))\cup \{u\}$ is an independent set with $\phi (S')>\phi (S)$, a contradiction.
\end{proof}

\noindent As a result of the above claim, we can deduce that $G[\overline{S}]$ contains no cycle.
 Moreover, since any interior vertex of a path in $G[\overline{S}]$ is a 3-vertex, we can distinguish only four  types of maximal paths in $G[\overline{S}]$ (see Figure~\ref{fig:types}): 
 \begin{itemize}
 \item A maximal path of length zero, and this type will be denoted by $\mathcal{P}_0$.

     \item A maximal path of length one and its end vertices are 2-vertices, and this type will be denoted by $\mathcal{P}_1$.
     \item A maximal path of length one and its end vertices are a 2-vertex and a 3-vertex, and this type will be denoted by $\mathcal{P}_2$.
     \item A maximal path of length two, and this type will be denoted by $\mathcal{P}_3$. 
 \end{itemize}
 Remark that any two vertices in $\overline{S}$, which are not on the same maximal path, are not adjacent. Moreover, by Claim~\ref{cl1.1}, the end vertices of a maximal path of type $\mathcal{P}_3$ are 2-vertices, while the interior one is a 3-vertex.\\

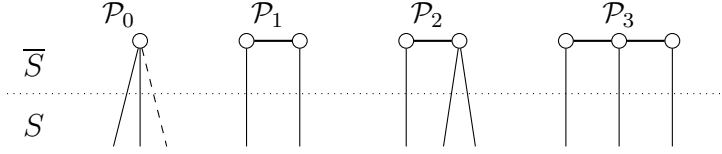
\begin{figure}[h]
\begin{center}
\begin{tikzpicture}[scale=.7]
\node at (-1,2.6) {$\overline{S}$};
\draw[dotted] (-1.5,2)--(12,2);
\node at (-1,1.4) {$S$};
\node at (1,3) (a) [circle,draw=black,fill=none,scale=0.5]{};\node at (.6,3.5) {\small $\mathcal{P}_0$};
\draw (1,1)-- (a) -- (0.5,1); \draw[dashed] (a) -- (1.5,1);

\node at (3,3) (b) [circle,draw=black,fill=none,scale=0.5]{};\node at (3.4,3.5) {\small $\mathcal{P}_1$};
\node at (4,3) (c) [circle,draw=black,fill=none,scale=0.5]{};
\draw (3,1)-- (b);  \draw[thick] (b) -- (c); \draw (c) -- (4,1); 

\node at (6,3) (d) [circle,draw=black,fill=none,scale=0.5]{};\node at (6.4,3.5) {\small $\mathcal{P}_2$};
\node at (7,3) (e) [circle,draw=black,fill=none,scale=0.5]{};
\draw (6,1)-- (d);  \draw[thick] (d) -- (e); \draw (7.3,1) -- (e) -- (6.7,1); 
 
\node at (9,3) (h) [circle,draw=black,fill=none,scale=0.5]{};\node at (10,3.5) {\small $\mathcal{P}_3$};
\node at (10,3) (i) [circle,draw=black,fill=none,scale=0.5]{};
\node at (11,3) (j) [circle,draw=black,fill=none,scale=0.5]{};
\draw (9,1)-- (h); \draw (i) -- (10,1); \draw[thick] (h) -- (i) -- (j); \draw (j) -- (11,1); 
\end{tikzpicture}
\end{center}
\caption{The different types of maximal paths in $G[\overline{S}]$.}
 \label{fig:types}   
\end{figure}

  \noindent The coloring procedure will be to color the vertices of $S$ by color $1_a$ and as many as possible vertices of $\overline{S}$ by color $1_b$ in such a way that there will remain at most one uncolored vertex for each maximal path of type $\mathcal{P}_1$, $\mathcal{P}_2$ and $\mathcal{P}_3$. We are going to use colors $3_a$ and $3_b$ for these vertices. 
  \\

\noindent In $G[\overline{S}]$, we call  {\em bad 2-vertex} each 2-vertex  on a maximal path of type $\mathcal{P}_1$, $\mathcal{P}_2$ and $\mathcal{P}_3$, and  {\em bad 3-vertex} each 3-vertex on a maximal path of type $\mathcal{P}_2$ and $\mathcal{P}_3$. A bad 3-vertex is said to be a {\em weak bad} 3-vertex if it is a vertex on a path of type $\mathcal{P}_2$, and {\em mid bad} 3-vertex if it is on a path of type $\mathcal{P}_3$. For abbreviation, we use bad vertex for any of the previously defined vertices. Let $u$ be a vertex in $G$, we call {\em bad neighbor} of $u$ every bad vertex adjacent to $u$. \\

\noindent For a maximum weighted independent set $T$, we denote by  $\theta(T)$  the number of maximal paths of type $\mathcal{P}_2$ in $G[\overline{T}]$. We will assume that our maximum weighted independent set $S$ was chosen such that  $\theta (S)\leq \theta (T)$ for every maximum weighted independent set $T$. \\

 \noindent If $u\in S$ is adjacent to $v\in \overline{S}$, then we say $u$ is a {\em father} of $v$. Two vertices $u$ and $v$ in $\overline{S}$ are said to be {\em siblings} if $uv \notin E(G)$, and they have a common father in $S$. In this case, we say $u$ is a sibling of $v$. Moreover,  we say $u$ is a {\em bad sibling} of $v$ if $u$ is a bad vertex. Since each sibling of a vertex $u$ is at distance two from $u$, we found that in order to study the distances between bad vertices, it is important to count the number of bad siblings of each bad vertex. We have the following result:
 \begin{claim}\label{cl1.2}
Each bad mid vertex has no bad sibling. Besides, each bad $2$-vertex and weak bad $3$-vertex has at most one bad sibling. 
 \end{claim}
 \begin{proof} 
 Let $u$ be a bad mid vertex and suppose that $u$ has a bad sibling $v$.  Clearly, $u$ has a unique father in $S$, say $x$.  Then, $x$ is a father of $v$. Moreover, $x$ is a 3-vertex since otherwise $S'=(S\setminus \{x\})\cup \{u\}$ is an independent set with $\phi (S')>\phi (S)$, a contradiction. Consequently,  $v$ is a 2-vertex. Thus, $S'=(S\setminus \{x\})\cup \{u,v\}$ is an independent set with $\phi (S')>\phi (S)$, a contradiction. \\
\noindent  We still need to prove each bad 2-vertex and weak bad 3-vertex has at most one bad sibling. Suppose to the contrary that there exists a bad $i$-vertex $u$, $i\in \{2,3\}$,  and two other bad vertices $v$ and $w$ such that $v$ and $w$ are siblings of $u$.  We will consider the cases concerning the nature of $u$:
 \begin{enumerate}
 \item $u$ is a 2-vertex.\\
 \noindent Since $u$ in this case has a unique father in $S$, then $u$, $v$ and $w$ have a common father, say $x$. As $x$ is a 3-vertex, then at most one of $v$ and $w$ is a 3-vertex. Note that if any of $v$ and $w$ have a neighbor in $S$ distinct from $x$ then this neighbor is a 2-vertex. Consequently, $S'=(S\setminus (N(v)\cup N(w))\cup \{u,v,w\}$ is an independent set with $\phi (S')>\phi (S)$, a contradiction. 
 
 \item $u$ is a 3-vertex.\\
 \noindent We need here to study two cases: $u$, $v$ and $w$ have a common father, and the case when the common father of $u$ and $v$ is distinct from that of $u$ and $w$. For the first case, let $x$ be the common father. Since  $x$ is a 3-vertex and $G$ is 1-saturated, then $v$ and $w$ are both 2-vertices and the other neighbor of $u$, distinct from $x$, in $S$ is also a 2-vertex. Consequently, $S'=(S\setminus N(u))\cup \{u,v,w\}$ is an independent set with $\phi (S')>\phi (S)$, a contradiction. For the other case, let $x$ be the common father of $u$ and $v$ and $y$ that of $u$ and $w$. Clearly, either $x$ or $y$ is a 2-vertex. Without loss of generality, suppose that $x$ is a 2-vertex. For the case $y$ is a 2-vertex and both $v$ and $w$ are 3-vertices such that the neighbor of $v$ (resp. $w$)  in $S$, distinct from $x$ (resp. $y$), is a 3-vertex, we get $S'=(S\setminus (N(v)\cup N(w)))\cup \{u,v,w\}$ is an independent set with $\phi (S')= \phi (S)$ but $\theta(S')< \theta(S)$, a contradiction. In fact, the neighbors of $v$ and $w$ in $S$, which are distinct from $x$ and $y$, are both bad mid vertices in $G[\overline{S'}]$ and this means that the maximal path containing each of these vertices in $G[\overline{S'}]$ is a path of type $\mathcal{P}_3$ and not $\mathcal{P}_2$, while the path to which $u$ (resp. $w$ and $v$) belongs in $S$ is a maximal path of type $\mathcal{P}_2$. For the case $y$ is a 3-vertex and $v$ is a 3-vertex such that $v$ has a neighbor in $S$, distinct from $x$, say $z$, which is a 3-vertex, we get $S' = (S \setminus (N(u) \cup N(v))) \cup \{u,v,w\}$ is an independent set with $\phi(S') = \phi(S)$ but $\theta (S') < \theta(S)$, a contradiction. In fact, the maximal path to which $z$ belongs in $G[\overline{S'}]$ is of type $\mathcal{P}_3$, while the maximal paths to which $u$ and $v$ belong in $G[\overline{S}]$ are of type $\mathcal{P}_2$. 
 For the remaining cases, whatever the nature of $v$, $w$ and $y$, we can prove in all cases that $S' = (S \setminus (N(v) \cup N(w))) \cup \{u,v,w\}$ is an independent set with $\phi(S') > \phi(S)$, a contradiction.
   \end{enumerate}\end{proof}\\

 \noindent  Let $u$ be a bad vertex having a sibling which is also a bad vertex, then $u$ is called a {\em sib.} For abbreviation, a {\em $2$-sib} (resp {\em $3$-sib}) is a  $2$-vertex sib (resp. $3$-vertex sib).  We call {\em bad set} of $\overline{S}$ each subset of $\overline{S}$  such that each of its vertices is either a bad 2-vertex which is not on a path of type $\mathcal{P}_3$ or a weak bad 3-vertex, and for every maximal path  $P$ of type $\mathcal{P}_1$ or  $\mathcal{P}_2$ in $G[\overline{S}]$, exactly one vertex of $P$ is in this set.   If $u$ is a bad $i$-vertex, $2\leq i\leq 3$, in a bad set $W$ and $u$ has no bad sibling in $W$, then $u$ is said to be a {\em lonely vertex of}  $W$. Otherwise, $u$ is said to be an {\em $i$-sib of $W$} (or {\em sib of W} for abbreviation). For a bad set $W$, we denote by $\gamma(W)$, the number of sibs of $W$.
  \noindent Let $B$ be a bad set of $\overline{S}$ such that $\gamma(B)$ is maximum. We have these important results concerning the distances between bad vertices and lonely vertices of $B$:\\

\begin{claim}\label{cl1.3}
Let $u$ and $v$ be two lonely vertices of $B$, then the bad neighbor of $u$ (resp.  $v$) in $\overline{S}$ has no sibling in $B$.  Moreover, the bad neighbor of $u$ is not a sibling of the bad neighbor of $v$ and $\dist(u,v)>3$. 
\end{claim}
\begin{proof}
Clearly, by the definition of the bad set and its lonely vertices, $u$ and $v$ have no common neighbor. Let $u'$ be the bad neighbor of $u$ and $v'$ be that of $v$ in $\overline{S}$. Suppose that $u'$ has a sibling in $B$, then $B'=(B\setminus\{u\})\cup \{u'\}$  is a bad set with $\gamma(B')> \gamma (B)$, a contradiction. Similarly, we can prove $v'$ has no sibling in $B$. Hence, a father of $u$  (resp. $v$) cannot be adjacent to $v'$ (resp. $u'$). Consequently, the distance between $u$ and $v$ is at least four. \\
\noindent We still need to prove that $u'$ is not a sibling of  $v'$. Suppose that $u'$ is a sibling of  $v'$, then $B'=(B\setminus\{u,v\})\cup \{u',v'\}$  is a bad set with $\gamma(B')> \gamma (B)$, a contradiction.
\end{proof}\\

\noindent Let $L$ be the set of lonely 3-vertices of $B$ and let $N=\{x:\; x\text{ is a bad neighbor of a vertex in  } L\}$. Then, each vertex in $N$ is a 2-vertex since $x$ is the other end of the path of type $\mathcal{P}_2$ in $G[\overline{S}]$ whose first end is in $L$. Moreover, $B'=(B\setminus L )\cup N$ is a bad set. By  Claim~\ref{cl1.3}, $\gamma(B')=\gamma (B)$ and so each vertex in $N$ is a lonely vertex of $B'$. It is important to notice that the main effect of defining $B'$ is that each lonely vertex of $B'$ is a 2-vertex. 

\begin{claim}\label{cl1.4}
We have $\dist(u,v)>3$ whenever $u$ and $v$ satisfy one of the following:
\begin{enumerate}
\item  $u$ and $v$ are two sibs of $B'$ with $u$ is not the sibling of $v$. 
\item $u$ and $v$ are two bad mid vertices in $\overline{S}$.
\item $u$ is a bad mid vertex and $v$ is a sib of $B'$.
\end{enumerate}
\end{claim}
\begin{proof}
\begin{enumerate}
\item  First, since $u$ and $v$ are not siblings and since they are not on the same maximal path, then $u$ and $v$ cannot have a common neighbor. By Claim~\ref{cl1.2}, a father of $u$ (resp. $v$) cannot be a father of the bad neighbor of $v$ (resp. $u$).  Thus, $\dist(u,v)>3$.
\item By Claim~\ref{cl1.2}, any bad mid vertex has no bad sibling. Thus,  any two bad mid vertices have no common neighbor, and the father of $u$ (resp $v$) has no common neighbor with $v$ (resp. $u$). Then, the result follows. 
\item  From  Claim~\ref{cl1.2},  we can deduce that $u$  and $v$ have no common neighbor and  a father of $u$ (resp. $v$)  cannot be adjacent  to the bad neighbors  of $v$ (resp. $u$). Hence, $\dist(u,v)>3$. 

\end{enumerate}
\end{proof}

\begin{claim}\label{cl1.5}
Let $u$ and $v$ be two sibs of $B'$ such that $u$ is a sibling of $v$. Then either $u$ or $v$ is at distance at least four from each lonely vertex of $B'$. 
\end{claim}
\begin{proof}Let $w$ be a lonely vertex of $B'$, then by Claim~\ref{cl1.2}, a father of $u$ (resp. $v$) cannot  be a father of $w$ and cannot be a father of the bad neighbor of $w$.\\
\noindent Let $u'$ (resp. $v'$) be the bad neighbor of $u$ (resp. $v$). Suppose to the contrary that both $u$ and $v$ are not at distance greater than three from each lonely vertex, then there exist two lonely vertices $u''$ and $v''$ such that $\dist(u,u'')<4$ and $\dist(v,v'')< 4$. Thus, the father of $u''$ (resp $v''$) is a father of $u'$ (resp. $v'$). We get $B''=(B'\setminus\{u,v\})\cup \{u',v'\}$ is a bad set but $\gamma(B'')>\gamma(B')$, a contradiction. In fact, $u'$, $v'$, $u''$ and $v''$ are all sibs of $B''$.\end{proof}\\
\noindent After we determined above some distances between vertices in $S$ and $\overline{S}$,  we are ready to define successively the desired partition  $C_1,C_2,C_3,C_4$ of $V(G)$:
\begin{enumerate}
  \item $C_1$ contains every lonely vertex of $B'$, and for every two sibs $u$ and $v$ of $B'$ with $u$ is a sibling of $v$, we have $|C_1\cap \{u,v\}|=1$ such that $x$ is at distance at least four from each lonely vertex of $B'$, where $\{x\}=C_1\cap \{u,v\}$ (such an $x$ exists by Claim~\ref{cl1.5}).
\item $C_2$ contains every  bad mid vertex in $\overline{S}$ and every sib $x$ of $B'$   such that the sibling of $x$ is in $C_1$.
\item  $C_3$ contains every vertex in $\overline{S}$ but not in $C_1\cup C_2$. 
\item $C_4=S$. 
\end{enumerate}
\noindent  By Claim~\ref{cl1.3}, Claim~\ref{cl1.4} (1), Claim~\ref{cl1.5},  and by the way of choosing the sibs of $B'$ to be in $C_1$, we obtain that the distance between any two vertices in  $C_1$ is at least four.  Besides,  by Claim~\ref{cl1.4}, we get that the distance between any two vertices in  $C_2$ is at least four.
Moreover, $C_3$ is an independent set.  In fact, any two vertices in $C_3$ are either not on the same maximal path or ends of a maximal path of type $\mathcal{P}_3$ and so they are not adjacent. Thus, we reached the desired partition, a contradiction.
\end{proof}

\section{  $(1,2,\dots, 2)$-packing coloring of subcubic graphs}
Before presenting the main results of this section, we are going to introduce the following lemmas:

\begin{lemma}\label{le1}
Let $G$ be a  $1$-saturated subcubic graph such that $G$ is not $(1,2^4)$-packing colorable with the minimum number of vertices. Then, no two $2$-vertices in $G$ are adjacent.
\end{lemma}
\begin{proof}
  Suppose to the contrary, there exist two adjacent 2-vertices, say $u$ and $v$. Let $w$ be the neighbor of $v$ distinct from $u$. If $wu\in E(G)$, then let $G'$ be the graph obtained from $G$ after deleting $v$. $G'$ is a 1-saturated subcubic graph and so $G'$ has a $(1,2^4)$-packing coloring, say $c$. We will prove that $c$ can be extended to a $(1,2^4)$-packing coloring of $G$, which is a contradiction. Since $u$ and $v$ are both 2-vertices, then either color 1 is not taken by the neighbors of $v$ or there exists $i$, $i\in \{a,b,c,d\}$, such that $v$ is at distance at least three from each vertex of color $2_i$. Then either we color $v$ by 1 or by $2_i$, and so we obtain a $(1,2^4)$-packing coloring of $G$, a contradiction.

\noindent For the case $uw\notin E(G)$, let $G'$ be the graph obtained from $G$ after deleting $v$ and then adding the edge $uw$. Again, $G'$ is a 1-saturated subcubic graph, and so  $G'$ has a $(1,2^4)$-packing coloring. Clearly, one (but not both since $u$ and $w$ are adjacent in $G'$) of the neighbors of $v$ is colored by 1, since otherwise, we give $v$ the color 1,  a contradiction. Let $x$ be the neighbor of $v$ of color 1, and let $y$ be the other neighbor of $v$. Without loss of generality, suppose $y$ is of color $2_a$.
 Moreover, the colors of the three colored neighbors of $u$ and $w$ are $2_b$, $2_c$, and $2_d$, since otherwise there exists $k\in \{b,c,d\}$ such that $v$ is at distance three from each vertex of color $2_k$, and so we give $v$ the color $2_k$, a contradiction.  Hence, $v$ is at distance at least three from each vertex of color $2_a$ except $y$. Since the colors of the three colored neighbors of $u$ and $w$ are $2_b$, $2_c$, and $2_d$, then $y$ has no neighbor of color 1. Consequently, recolor $y$ by 1 and then  color  $v$ by $2_a$, and so we obtain a $(1,2^4)$-packing coloring of $G$, a contradiction.
\end{proof}

\noindent  Recall that a $(3,0)$-saturated subcubic graph is a subcubic graph on which every two heavy vertices are not adjacent.
 \begin{lemma}\label{le2}
Let $G$ be a  $(3,0)$-saturated subcubic graph such that $G$ is not $(1,2^5)$-packing colorable with the minimum number of vertices. Then, no two $2$-vertices in $G$ are adjacent.
\end{lemma}
\begin{proof}
    We can proceed as in the proof of Lemma~\ref{le1}, since the above defined subgraph $G'$ is  $(3,0)$-saturated whenever $G$ is $(3,0)$-saturated.
\end{proof}
\begin{theorem}\label{thm4}
Every $1$-saturated subcubic graph is $(1,2^4)$-packing colorable.
\end{theorem}
\begin{proof}
On the contrary, suppose that $G$ is a counter-example with the minimum number of vertices.  
By Lemma~\ref{le1}, any two 2-vertices in $G$ are not adjacent.
 Thus, the set of the 2-vertices in $G$ is independent. Moreover, since $G$ is 1-saturated, then each 3-vertex  is at distance at least three from each other 3-vertex but at most three.  Thus, color each 2-vertex by 1, then color greedily the 3-vertices  by the colors $2_a$, $2_b$, $2_c$ and $2_d$ in such a way any two vertices receiving the same color $2_i$ are at distance at least three from each other. Hence, we obtain a $(1,2^4)$-packing coloring of $G$, a contradiction.
\end{proof}

Observe that the above result is tight in the sense that the graph depicted on Figure~\ref{Fig:1sat} is subcubic and $1$-saturated and not $(1,2^3)$-packing colorable.

\begin{figure}[h]
\begin{center}
\begin{tikzpicture}[scale=1.3]
\node at (0,0) (a) [circle,draw=black,fill=none,scale=0.7]{};
\node at (1,0) (b)[circle,draw=black,fill=none,scale=0.7]{}; 
\node at (2,0) (c) [circle,draw=black,fill=none,scale=0.7]{}; 
\node at (2,1) (d) [circle,draw=black,fill=none,scale=0.7]{}; 
\node at (2,2) (e) [circle,draw=black,fill=none,scale=0.7]{}; 
\node at (1,2) (f)[circle,draw=black,fill=none,scale=0.7]{}; 
\node at (0,2) (g) [circle,draw=black,fill=none,scale=0.7]{}; 
\node at (0,1) (h) [circle,draw=black,fill=none,scale=0.7]{}; 

\draw (e) -- (a) -- (b) -- (c) -- (d) -- (e) -- (f) -- (g) -- (h) -- (a);
\draw (c) -- (g);
\end{tikzpicture}
\end{center}
\caption{A $1$-saturated non $(1,2^3)$-packing coloring subcubic graph.}
\label{Fig:1sat}
\end{figure}
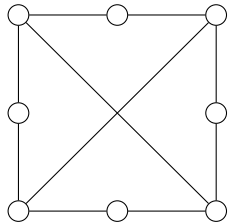

\begin{theorem}\label{thm5}
    Every $(3,0)$-saturated subcubic graph is $(1,2^5)$-packing colorable.
\end{theorem}
\begin{proof}
  On the contrary, suppose that $G$ is a counter-example with the minimum order $n$.  Clearly, $G$ is connected. We have the following result concerning the neighbors of a 3-vertex:
  \begin{claim}\label{cl3.1}
      Every 3-vertex is adjacent to at most one 2-vertex.
  \end{claim}
     \begin{proof} Suppose, to the contrary, that there exists a 3-vertex having two neighbors of degree two, and let $x$ be one of these two neighbors. By Lemma~\ref{le2}, every neighbor of $x$ is a 3-vertex. Let $u$ and $v$ be the two neighbors of $x$ such that $u$ has a neighbor of degree two distinct from $x$. Let $G'=G\setminus \{x\}$, then, by the minimality of the order of $G$, $G'$ has a $(1,2^5)$-packing coloring. First, $u$ and $v$ are not adjacent since otherwise there exists $j$, $1\leq j\leq 5$, such that $x$ is at distance at least three from each vertex of color $2_j$, then give $x$ the color $2_j$ and so we obtain a  $(1,2^5)$-packing coloring of $G$, a contradiction. If both neighbors of $x$ are colored by 1, then there exists $j$, $1\leq j\leq 5$, such that $x$ is at distance at least three from each vertex of color $2_j$, then color $x$ by $2_j$, and so we obtain a  $(1,2^5)$-packing coloring of $G$, a contradiction.  Suppose now that  $x$ has only one neighbor of color 1 and let us call this neighbor $y$. Consequently, the neighbor of $x$, which is distinct from $y$, has a neighbor of color 1, since otherwise we recolor it by 1 and then proceed as in the previous case, when both neighbors of $x$ are of color 1,  to get a contradiction.  Thus,  there exists $j$, $1\leq j\leq 5$, such that $x$ is at distance at least three from each vertex of color $2_j$, then color $x$ by $2_j$, a contradiction.  Moreover, if $x$ has no neighbor of color 1 and if the two neighbors of $x$ are of distinct colors, then color $x$ by 1, a contradiction. We are left with the case when $u$ and $v$ have the same color $2_j$ for some $j$, $1\leq j\leq 5$. The problem here is that $u$ and $v$ are at distance two from each other in $G$, but both are colored by $2_j$. If $u$  has no neighbor of color 1, then recolor $u$  by 1 and so we proceed as before in order to get a contradiction. Similarly, $v$ has a neighbor of color 1. Let $u_1$ and $u_2$ be the two neighbors of $u$ distinct from $x$, and suppose $u_1$ is the one of color 1. Recall that $u$ has a neighbor of degree two distinct from $x$. If both $u_1$ and $u_2$ are of color 1, then there exist $i$, $1\leq i\neq  j\leq 5$, such that $u$ is at distance at least three from each vertex of color $2_i$ and so we recolor $u$ by $2_i$ and color $x$ by 1,  a contradiction. Hence, we can deduce $u_2$ is not of color 1 and it has a neighbor of color 1, since otherwise we recolor $u_2$ by 1 and we proceed as before to reach a contradiction. 
     But again, there exists $i$, $1\leq i\neq j\leq 5$, such that $u$ is at distance at least three from each vertex of color $2_i$. Thus, we can recolor $u$ by $2_i$ and color $x$ by 1,  and so we obtain a  $(1,2^5)$-packing coloring of $G$, a contradiction.
      \end{proof}
      

 \noindent Let $X$ be the set of heavy vertices and 2-vertices in $G$ and let $\overline{X}=V(G)\setminus X$. Since $G$ is $(3,0)$-saturated and by Lemma~\ref{le2}, $X$ is an independent set. Let $G'$ be the graph whose set of vertices is   $\overline{X}$ such that two vertices $x$ and $y$ are adjacent in $G'$ if and only if the distance between  $x$ and $y$ is at most two in $G$. Remark that if we prove the vertices of $G'$ can be colored properly by five colors, then the vertices of  $\overline{X}$ can be colored by $2_1,\dots,2_5$, and so if we give then the vertices of $X$ the color 1, we get a $(1,2^5)$-coloring of $G$, a contradiction. \\
 \noindent Hence $G'$ cannot be colored properly by five colors. However, it is easy to notice that each non-heavy vertex is at distance at least three from each other non-heavy vertex in $G$  but at most five. Consequently, $\Delta(G')\leq 5$. Since $G'$ cannot be colored properly by five colors, then, by Brook's theorem, $G'$ has a complete subgraph on six vertices, say $K$. Let $\{x_1, \dots, x_6\}$ be the vertices of $K$. \\
 \noindent We will study first the case when there exists a vertex in $K$, say $x$, such that $x$ is adjacent in $G$ to two vertices in $\{x_1,\dots,x_6\}$. Without loss of generality, suppose $x=x_1$ and suppose $x_2$ and $x_3$ are neighbors of $x_1$ in $G$, and so the other vertices of $K$  are second neighbors in $G$ of $x_1$. Let $y_1$ (resp. $y_2$ and $y_3$) be the neighbor of $x_1$ (resp. $x_2$ and $x_3$) of degree two in $G$. Clearly $y_1$ is adjacent to one of $x_4$, $x_5$ and $x_6$. Without loss of generality, suppose it is  $x_4$, and then suppose $x_5$ (resp. $x_6$) is adjacent to $x_2$ (resp. $x_3$). 
 \noindent The only way for $x_4$ to be at distance less than three from $x_2$ (resp. $x_3$) is to have a common neighbor with $x_2$ (resp. $x_3$). In fact,  $x_2$ (resp. $x_3$) has three  neighbors  which are $x_1$,  $x_5$ and $y_2$ (resp. $x_1$,  $x_6$ and $y_3$). 
 Thus, the only possible common neighbor of $x_4$ and $x_2$ (resp. $x_4$ and $x_3$) is $x_5$ (resp. $x_6$), since by  Claim~\ref{cl3.1}, $x_4$ has only one neighbor of degree two which is $y_1$. However, the only way now for $x_2$ and $x_6$ (resp. $x_3$ and $x_5$) to be at distance at most two in $G$ is by having a common neighbor since each of them is non-heavy and already adjacent to two 3-vertices. Therefore,  the only possible common neighbor of $x_2$ and $x_6$ (resp. $x_3$ and $x_5$) is $y_2$ (resp. $y_3$). Hence, $G$ is the graph on the left of Figure~\ref{Fig1} which is $(1,2^5)$-packing colorable (even $(1,2^4)$-packing colorable), a contradiction.\\
 \noindent Thus, every vertex in $K$ is adjacent in $G$ to at most one vertex in $K$. Suppose there exists a vertex in $K$, say $x$, such that $x$ is adjacent in $G$ to a vertex in $\{x_1,\dots,x_6\}$. Without loss of generality, suppose $x=x_1$ and suppose $x_2$  is a neighbor of $x_1$ in $G$, and so the other vertices of $K$  are second neighbors in $G$ of $x_1$. Since every vertex in $K$ is adjacent in $G$ to at most one vertex in $K$, then none of the neighbors of $x_2$ in $G$ is in $\{x_3,\dots,x_6\}$. But $x_1$ is adjacent to a 2-vertex, then there exists $k$, $3\leq k\leq 6$, such that $x_k$ is not a second neighbor of $x_1$ in $G$, a contradiction.\\
 \noindent Consequently, for every two vertices $x$ and $y$ in $K$, $x$ and $y$ are not adjacent in $G$. Let $y_1$, $y_2$ and $y_3$ be the neighbors of $x_1$ in $G$ such that $y_1$ is a 2-vertex. Remark that both $y_2$ and $y_3$ are heavy vertices, since $\Delta(G')\leq 5$ and $x_1$ has five neighbors in $G'$ distinct from $y_2$ and $y_3$. Without loss of generality, suppose $x_2$ is a neighbor of $y_1$, $x_3$ and $x_4$ are both neighbors of $y_2$, and $x_5$ and $x_6$ are both neighbors of $y_3$ in $G$.   Since the only way for every two vertices in $\{x_2,\dots, x_6\}$  to be at distance less than three from each other in $G$ is to have a common neighbor, we get that $G$ is the graph on the right of Figure~\ref{Fig1}, which is $(1,2^5)$-packing colorable (even $(1,2^4)$-packing colorable), a contradiction.
\end{proof}\\

\begin{figure}[h]
\begin{center}
\begin{tikzpicture}[scale=1.3]
\node at (1,5) (a) [circle,draw=black,fill=none,scale=0.7]{$x_1$};
\node at (.6,5) {\small$1$};
\node at (2.75,5) (b) [circle,draw=black,fill=none,scale=0.7]{$x_3$};
\node at (3.2,5) {\small$2_1$};
\node at (3.5,3.5) (c)[circle,draw=black,fill=none,scale=0.7]{$x_6$};
\node at (3.9,3.5) {\small$1$};
\node at (2.75,2) (d) [circle,draw=black,fill=none,scale=0.7]{$x_4$};
\node at (3.2,2) {\small$2_2$};
\node at (1,2) (e) [circle,draw=black,fill=none,scale=0.7]{$x_5$};
\node at (.6,2) {\small$1$};
\node at (.25,3.5) (f)[circle,draw=black,fill=none,scale=0.7]{$x_2$};
\node at (-.2,3.5) {\small$2_3$};
\node at (2.5,3.7) (g) [circle,draw=black,fill=none,scale=0.7]{$y_1$};\node at (2.5,4.1) {\small$2_4$};
\node at (1.3,3.7) (h) [circle,draw=black,fill=none,scale=0.7]{$y_3$};\node at (1.3,4.1) {\small$2_4$};
\node at (1.9,2.8) (i) [circle,draw=black,fill=none,scale=0.7]{$y_2$};\node at (1.9,2.4) {\small$2_4$};
\draw (a) -- (b) -- (c) -- (d) -- (e) -- (f) -- (a);
\draw (a) -- (g) -- (d);
\draw (b) -- (h) -- (e);
\draw (c) -- (i) -- (f);

\node at (7,5.2) (a) [circle,draw=black,fill=none,scale=0.7]{$x_1$}; \node at (7.4,5.2) {\small$1$};
\node at (5.5,4.1) (g) [circle,draw=black,fill=none,scale=0.7]{$y_1$};\node at (5.9,4) {\small$2_1$};
\node at (7,4.1) (h) [circle,draw=black,fill=none,scale=0.7]{$y_2$};\node at (7.4,4) {\small$2_2$};
\node at (9,4.1) (i) [circle,draw=black,fill=none,scale=0.7]{$y_3$};\node at (9.4,4) {\small$2_3$};
\node at (5.5,2.8) (b)[circle,draw=black,fill=none,scale=0.7]{$x_2$}; \node at (5.9,3) {\small$2_2$};
\node at (6.5,2.8) (c) [circle,draw=black,fill=none,scale=0.7]{$x_3$}; \node at (6.8,3) {\small$1$};
\node at (7.5,2.8) (d) [circle,draw=black,fill=none,scale=0.7]{$x_4$}; \node at (7.9,3) {\small$2_3$};
\node at (8.5,2.8) (e) [circle,draw=black,fill=none,scale=0.7]{$x_5$}; \node at (8.9,3) {\small$2_4$};
\node at (9.5,2.8) (f)[circle,draw=black,fill=none,scale=0.7]{$x_6$}; \node at (9.8,3) {\small$1$};
\node at (7,1.5) (j) [circle,draw=black,fill=none,scale=0.7]{}; \node at (6.8,1.3) {\small$2_4$};
\node at (6,1.5) (k) [circle,draw=black,fill=none,scale=0.7]{}; \node at (5.8,1.3) {\small$1$};
\node at (8,1.5) (l) [circle,draw=black,fill=none,scale=0.7]{}; \node at (7.8,1.3) {\small$2_1$};
\node at (9,1.5) (m)[circle,draw=black,fill=none,scale=0.7]{}; \node at (8.8,1.3) {\small$2_1$};
\draw (g) -- (a) -- (h);
\draw (a) -- (i);
\draw (g) -- (b);
\draw (c) -- (h) -- (d);
\draw (f) -- (i) -- (e);
\draw (b) -- (j) -- (c); \draw (j) -- (f);
\draw (b) -- (k) -- (d); \draw (k) -- (e);
\draw (c) -- (l) -- (e);
\draw (d) -- (m) -- (f);
\end{tikzpicture}
\end{center}
\caption{Two configurations of 6 non-heavy vertices $x_i$ at pairwise distance at most two in a $(3,0)$-saturated subcubic graph, along with a $(1,2^4)$-packing coloring for each.}
\label{Fig1}
\end{figure}
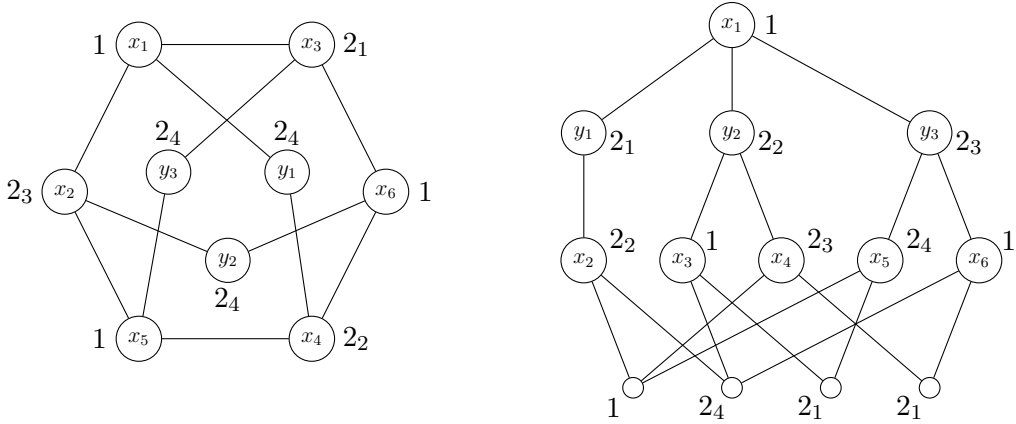

\noindent Remark that both graphs  of Figure~\ref{Fig1} are not $(1,2^3)$-packing colorable. For the one on the left, we can observe that the six 3-vertices are at pairwise distance at most two and that the three 2-vertices are at distance at most 2 of any 3-vertex. Hence, in order to define a $(1,2^3)$-packing coloring, one must give color 1 to half of the six 3-vertices. But then, it can be seen that each 2-vertex is adjacent to a vertex of color 1, and thus it can be colored neither by color 1 nor by a color 2. For the graph on the right, this was confirmed by a computer exhaustive search.


\section{Concluding Remarks}
\noindent Remark that the result of Theorem~\ref{thm1} is maybe not tight since we were only able to find a 1-saturated subcubic graph that is not $(1,1,4,4)$-packing colorable. The graph on the left of Figure~\ref{fig2} has this property. Actually, it can be observed that, while the diameter of this graph is 5, the distance between two vertices lying in a triangle is at most 4. Hence, it is impossible to complete the coloring of the three triangles with only two colors 4. \\

\noindent Thus, we propose the following problem.

\noindent \textbf{Open problem }: Is it possible to use the method of the proof of Theorem~\ref{thm1} for proving that 1-saturated subcubic graphs are $(1,1,3,4)$-packing colorable? And if yes, what values of $\alpha$ and $\beta$  in $|X_1(T)|+\alpha |X_0(T)|+\beta |Y(T)|$ must be used?\\

\end{document}